\documentclass[11pt]{article}
\usepackage[top=1.25in, bottom=1.5in, left=1.5in, right=1.5in]{geometry}
\usepackage{amsmath,amsthm,amssymb}

\newtheorem{theorem}{Theorem}[section]
\newtheorem{corollary}[theorem]{Corollary}
\newtheorem{lemma}[theorem]{Lemma}
\newtheorem{example}[theorem]{Example}
\newtheorem*{remark}{Remark}
\newtheorem*{remarks}{Remarks}
\newtheorem{construction}{Construction}[section]

\renewcommand{\P}{\ensuremath{\mathcal{P}}}
\newcommand{\A}{\ensuremath{\mathcal{A}}}
\newcommand{\B}{\ensuremath{\mathcal{B}}}
\newcommand{\G}{\ensuremath{\mathcal{G}}}
\newcommand{\zed}{\ensuremath{\mathbb{Z}}}

\title{A method of constructing pairwise balanced designs containing parallel classes}
\author{Douglas R.\ Stinson\thanks{D.R.\ Stinson's research is supported by  NSERC discovery grant RGPIN-03882.}\\
David R.\ Cheriton School of Computer Science\\University of Waterloo\\ Waterloo ON, N2L 3G1, Canada}
\date{\today}

\begin{document}

\maketitle

\begin{abstract}
The obvious way to construct a GDD (group-divisible design) recursively is to use {\sc Wilson's Fundamental Construction} for GDDs ({\sc WFC}). Then a PBD (pairwise balanced design) is often obtained by adding a new point to each group of the GDD. However, after constructing such a PBD, it might be the case that we then want to identify a parallel class of blocks. In this short note, we explore some possible ways of doing this.
\end{abstract}

\section{Introduction and Definitions}

We use standard design-theoretic terminology for GDDs (group-divisible designs), PBDs (pairwise balanced designs)  and
transversal designs (TDs). To begin, we recall a few definitions of these kinds of designs from \cite{CD}.

Let $K$ and $L$ be sets of positive integers (we can assume that every element of $K$ is at least two). A \emph{$K$-group-divisible design} (or \emph{$K$-GDD}) with group sizes in $L$  is a triple $(X,\G,\A)$ that satisfies the following properties:
\begin{enumerate}
\item $X$ is a set of  \emph{points}.
\item $\G$ is a partition of $X$ into \emph{groups} such that $|G| \in L$ for all $G \in \G$.
\item $\A$ consists of a set of \emph{blocks} such that 
\begin{enumerate}
\item $|G \cap A| \leq 1$ for all $G \in \G$ and for all $A \in \A$,
\item every pair of points from different groups is contained in exactly one block, and
\item $|A| \in K$ for all $A \in \A$.
\end{enumerate}
\end{enumerate}
If $K = \{k\}$, we write $k$-GDD for simplicity. The \emph{type} of the GDD is the multiset $\{|G|: G \in \G\}$. We usually employ an exponential notation to describe types:
type ${t_1}^{u_1} {t_2}^{u_2}  \dots$ denotes $u_i$ occurrences of $t_i$ for $i = 1,2, \dots$. 

A \emph{TD$(k,m)$} (or \emph{transversal design})is a $k$-GDD of type $m^k$. Thus every block in a transversal design is a transversal of the $k$ groups. It is well-known that a TD$(k,m)$ is equivalent to $k-2$ mutually orthogonal latin squares of order $m$. 

Let $K$ be a set of positive integers. A \emph{$(v,K)$-pairwise balanced design} (or \emph{$(v,K)$-PBD}) is a pair $(X,\A)$ that satisfies the following properties:
\begin{enumerate}
\item $X$ is a set of  \emph{points},
\item $\A$ consists of a set of \emph{blocks} such that 
every pair of points is contained in exactly one block, and
\item $|A| \in K$ for all $A \in \A$.
\end{enumerate}
If $K = \{k\}$, we write $(v,k)$-PBD for simplicity. A $(v,k)$-PBD is also known as a \emph{$(v,k,1)$-balanced incomplete block design} (or \emph{$(v,k,1)$-BIBD}).

If $q$ is a prime or prime power, then there exists a $(q^2,q,1)$-BIBD (an \emph{affine plane} of order $q$) and a $(q^2+q+1,q+1,1)$-BIBD 
(a \emph{projective plane} of order $q$).

\medskip

Before describing our constructions for PBDs containing parallel classes, we recall \textsc{Wilson's Fundamental Construction} for GDDs (which we abbreviate to \textsc{WFC}). We follow the presentation from \cite[\S IV.2.1]{CD}.

\begin{construction}[\textsc{Wilson's Fundamental Construction for GDDs}]
\label{WFC.const}
Suppose that $(X,\G,\A)$ is a GDD and let $w : X \rightarrow \zed^+$ ($w$ is called a \emph{weighting}). For every block $A \in \A$, suppose there is a
$K$-GDD of type $\{ w(x) : x \in A\}$. For all $G \in \G$, define $w_G = \sum_{x \in G} w(x)$. Then there is a $K$-GDD of type
$\{ w_G : G \in \G\}$. 
\end{construction}

\section{Constructions}

We often construct GDDs recursively using  Construction \ref{WFC.const}. Then a PBD can be obtained from the resulting GDD by adding a new point to each group. 
Suppose, after constructing such a PBD, that we then want to identify a parallel class of blocks. Our next theorem provides one way to accomplish this goal.

\begin{theorem}
\label{t1}
Suppose there is a TD$(\ell+1,m)$ and a TD$(\ell,u)$, where $u \leq m$. Suppose $\ell \in K$ and suppose that there is a $K$-GDD of type $u^{\ell} v^1$. 
Finally, suppose there exists an $(mu+1,K)$-PBD. Then there exists a 
$K$-GDD of type $\ell^{mu}{(tv+1)}^1$ for all $t$ such that $0 \leq t \leq m - u$. 
\end{theorem}

\begin{proof}
Start with a TD$(\ell+1,m)$ and delete $m-t$ points from the last group. 
This yields an $(\{\ell, \ell+1\},1)$-GDD of type $m^{\ell} t^1$. 
The blocks have sizes $\ell$ and $\ell+1$ and every block of size $\ell+1$ intersects the last group. Give every point in the first $\ell$ groups weight $u$, give every point in the last group weight $v$, and apply Construction \ref{WFC.const} (\textsc{WFC}). 
For a block of size $\ell$, we fill in a TD$(\ell,u)$.
For a block of size $\ell+1$, we fill in a $K$-GDD of type $u^{\ell} v^1$.

We now have a $K$-GDD of type $(mu)^{\ell} (tv)^1$.  Let $\infty$ be a new point.  
Replace every group $G$ of size $mu$ by an $(mu+1,K)$-PBD on $G \cup \{\infty\}$. Also, add $\infty$ to the last group. 
This produces an $(mu\ell + tv+1,K\cup \{tv+1\})$-PBD.

It remains to identify a parallel class of blocks in this PBD. The parallel class will consist of $mu$ blocks of size $\ell$ and the block of size $tv+1$.
Choose $u$ of the points that were deleted from the last group of the TD$(\ell +1, m)$ (note that $u \leq m-t$). These $u$ points induce $u$ classes of $m$ blocks of size $\ell$, each of which partitions the points in the first $\ell$ groups of the TD$(\ell +1, m)$. Denote these classes by $\P_i$, $1 \leq i \leq u$. 

When we apply \textsc{WFC}, we replace every point $x$ by a set of $u$ points, say $\{x\} \times \{1,\dots , u\}$.
Every block $B$ of size $\ell$ is replaced by the $u^2$ blocks in a TD$(\ell,u)$, in which the groups are
$\{x\} \times \{1,\dots , u\}$, $x \in B$. 
For all $B \in \P_i$ (where $1 \leq i \leq u$), we can stipulate that $B' = \{ (x,i) : x \in B\}$ is one of the blocks in the 
TD$(\ell,u)$ constructed from $B$. 

Now define 
\[ \P = \{ B' : B \in \P_i, 1 \leq i \leq u\}.\]
It is easily seen that $\P$ is a set of $mu$ blocks of size $\ell$ that form a partition of the first $\ell$ groups of the 
$K$-GDD of type $(mu)^{\ell} (tv)^1$. In the constructed PBD, there is a unique block $B_0$ of size $tv+1$ arising  from the last group of the TD$(\ell+1,m)$ together with  $\infty$. The blocks in $\P$, along with $B_0$, comprise the desired parallel class. This parallel class is taken to be the groups in a $K$-GDD of type $\ell^{ mu}{(tv+1)}^1$. \end{proof}

Here is a specific application of Theorem \ref{t1}.
\begin{corollary}
\label{cor2}
Suppose $m \equiv 0$ or $1 \bmod 5$, $m > 10$, 
and let $0 \leq t \leq m-4$.
Then there exists a $5$-GDD of type $5^{4m}{(4t+1)}^1$.
\end{corollary}

\begin{proof}
We apply Theorem \ref{t1} with $\ell = 5$, $u = v = 4$ and $K = \{5\}$. A $K$-GDD of type $u^{\ell} v^1$ is just a
$5$-GDD of type $4^6$, which is obtained from an affine plane of order $5$ with a point deleted. A TD$(\ell,u)$ is obtained from a projective plane of order $4$ by deleting a point. 
A TD$(6,m)$ exists from \cite[\S III.3.6]{CD}. An $(mu+1,K)$-PBD is just a 
$(4m+1,5,1)$-BIBD, which exists because $m \equiv 0$ or $1 \bmod 5$ (see \cite[\S II.3.1]{CD}). We obtain a $5$-GDD of type $5^{4m}{(4t+1)}^1$.
\end{proof}

\begin{remark} {\rm The constructed PBD has blocks of size five and a block of size $4t+1$. Many results on such PBDs are known, e.g., see \cite[\S IV.1.2]{CD}. But there is apparently less information known on when such a PBD contains a parallel class that includes the block of size $4t+1$. }
\end{remark}

Here is one small numerical example to illustrate. 

\begin{example}
We construct  $5$-GDDs of type $5^{44}s^1$ for $s = 1,5, \dots , 29$. Take $m = 11$, let $t = 0,1, \dots , 7$ and  
apply Corollary \ref{cor2}.
\end{example}

We next observe that we can improve Theorem \ref{t1} if we have some information about the existence of disjoint blocks in the TD$(\ell,u)$.

\begin{theorem}
\label{t3}
Suppose there is a TD$(\ell+1,m)$. Suppose also that there is a TD$(\ell,u)$ containing $\alpha$ disjoint blocks, 
where $u \leq m$. Suppose $\ell \in K$ and suppose that there is a $K$-GDD of type $u^{\ell} v^1$. 
Finally, suppose there exists an $(mu+1,K)$-PBD. Then there exists a 
$K$-GDD of type $\ell^{mu}{(tv+1)}^1$ for all $t$ such that $0 \leq t \leq m - \lceil u / \alpha\rceil$. 
\end{theorem}

\begin{proof}
The construction of the $(mu\ell + tv+1,\{K, (tv+1)^*)$-PBD is the same as in the proof of Theorem \ref{t1}. However, we construct the parallel class of blocks of size $\ell$ slightly differently. Choose $\lceil u / \alpha\rceil$ of the points that were deleted from the last group of the TD$(\ell +1, m)$ (note that $\lceil u / \alpha\rceil \leq m-t$). These $\lceil u / \alpha\rceil$ points induce 
$\lceil u / \alpha\rceil$ classes of $m$ blocks of size $\ell$, each of which partitions the points in the first $\ell$ groups of the TD$(\ell +1, m)$. Denote these classes by $\P_i$, $1 \leq i \leq \lceil u / \alpha\rceil$. 

When we apply \textsc{WFC}, we replace every point $x$ by a set of $u$ points, say $\{x\} \times \{1,\dots , u\}$. 
Every block $B \in \P_i$ is replaced by the $u^2$ blocks in a TD$(\ell,u)$ in which the groups are
$\{x\} \times \{1,\dots , u\}$, $x \in B$. 

Partition the set $\{1,\dots , u\}$ into $\lceil u / \alpha\rceil$ disjoint sets, say 
$T_1, \dots , T_{\lceil u /\alpha\rceil}$, each of size at most $\alpha$.
Each TD$(\ell,u)$ contains $\alpha$ disjoint blocks. For $1 \leq i \leq \lceil u / \alpha\rceil$, for each $B \in \P_i$ and every 
$j \in T_i$,
define $B'_j = \{ (x,j) : x \in B\}$. 
We can stipulate that the $\alpha$ blocks $B'_j$ (for all $j \in T_i$) are blocks in the TD$(\ell,u)$ constructed from $B$.

Now define 
\[ \P = \{ B'_j : B \in \P_i, 1 \leq i \leq \lceil u /\alpha\rceil, j \in T_i\}.\]
It is easily seen that $\P$ is a set of $m u$ blocks of size $\ell$ that form a partition of the first $\ell$ groups of the 
$K$-GDD of type $(mu)^{\ell} (tv)^1$. The rest of the construction proceeds as before.
\end{proof}

\begin{remark} {\rm The advantage of  Theorem \ref{t3} (with $\alpha > 1$) as compared to Theorem \ref{t1} is that we can take larger values of $t$ in Theorem \ref{t3}.}
\end{remark}

In order to apply Theorem \ref{t3}, we  need to know something about disjoint blocks in a TD$(\ell,u)$.
This problem has been  studied extensively in the case $\ell = 3$, where a set of disjoint blocks is a partial transversal of the associated latin square. See Wanless \cite{Wan} for a survey of results on this problem.
We briefly mention a few general results for arbitrary $\ell$ that are well-known and/or follow from elementary counting arguments.


\begin{lemma}
\label{l4}
\mbox{\quad} \vspace{-.2in}\\
\begin{enumerate}
\item If there is a TD$(\ell,u)$, then $\ell \leq u+1$.
\item A TD$(u+1,u)$ does not contain two disjoint blocks.
\item A TD$(u,u)$ contains $u$ disjoint blocks.
\item If there is a TD$(\ell+1,u)$, then there is a TD$(\ell,u)$ that contains $u$ disjoint blocks.
\item A TD$(\ell,u)$ with $u \geq \ell \geq 2$ contains at least three disjoint blocks, unless $u = \ell = 2$.
\item A TD$(\ell,u)$ contains at least $\left\lceil \frac{u^2}{\ell (u-1) + 1} \right\rceil $ disjoint blocks.
\end{enumerate}
\end{lemma}

\begin{proof}
Parts 1.--4.\ are well-known, so we only provide a proof of parts 5 and 6. 
First we prove part 5. A block $B$ in a TD$(\ell,u)$ intersects $\ell (u-1)$ other blocks. 
There are $u^2$ blocks. Hence, there exists a block disjoint from
$B$ if and only if $u^2 > 1 + \ell(u-1)$, or $\ell < u + 1$. Since $\ell \leq u$, there are at least two disjoint blocks.
Now, assume that $B_1$ and $B_2$ are disjoint blocks. There are $\ell(\ell-1)$ blocks that intersect both $B_1$ and $B_2$.
Since every point is in $u$ blocks, there are $2\ell(u - \ell)$ blocks that contain exactly one point from $B_1 \cup B_2$. 
It follows that there is a block disjoint from both $B_1$ and $B_2$ if and only if
$u^2 - 2\ell(u - \ell)- \ell(\ell-1) - 2 > 0$. Fix $\ell$ and define \[f(u) = u^2 - 2\ell(u - \ell)- \ell(\ell-1) - 2.\]
We have $f'(u) = 2(u - \ell) \geq 0$ if $u \geq \ell$. Also, $f(\ell) = \ell - 2 > 0$ since $\ell > 2$. It follows that
$f(u) > 0$ for all $u \geq \ell$ when $\ell \geq 3$. When $\ell = 2$, we have $f(u) = (u-2)^2$, so $f(u) > 0$ if and only if $u > 2$.
This establishes the existence of three disjoint blocks in the TD, unless $u = \ell = 2$.

To prove part 6, let $B_1, \dots , B_r$ be a maximal set of $r$ disjoint blocks in a TD$(\ell,u)$.  
Denote $Y = \bigcup_{i = 1}^{r} B_i$. 
Since we started with a maximal set of disjoint blocks, there is no block disjoint from $Y$.
Denote the set of $u^2 - r$ blocks other than $B_1, \dots , B_r$ by $\B'$.
For $B \in \B'$, define $a_B = | B \cap Y|$. By assumption, $a_B \geq 1$ for all $B \in \B'$. We have
\[ \sum_{B \in \B'} a_B = r\ell(u-1).\]
Therefore the mean of the $a_B$'s is 
\[\overline{a} = \frac{r\ell(u-1)}{u^2 - r}.\]
Since $\overline{a} \geq 1$, we have 
\[\frac{r\ell(u-1)}{u^2 - r} \geq 1,\]
or  \[ r\ell(u-1) \geq u^2 - r.\]
Consequently, \[r \geq \frac{u^2}{\ell (u-1) + 1}.\]
Hence, the TD$(\ell,u)$ contains at least $\left\lceil \frac{u^2}{\ell (u-1) + 1} \right\rceil $ disjoint blocks.
\end{proof}


\begin{remarks}
{\rm \quad \vspace{-.22in}\\
\begin{enumerate}
\item Wilson's 1974 construction for mutually orthogonal latin squares \cite{W} explicitly makes use of transversal designs containing disjoint blocks. Also, it is observed in  \cite{W} that a TD$(\ell,u)$ contains at least two disjoint blocks if $\ell \leq u$. 
\item We observe that Corollary \ref{cor2} is a special case of Theorem \ref{t3} in which $\alpha = 1$. In view of Lemma \ref{l4}, we cannot take $\alpha > 1$ in this case because $\ell = u+1$.
\item The bound proven in part  6 also follows from a more general result due to Rosenfeld \cite[Theorem I b)]{Rosenfeld}.
\end{enumerate}
}
\end{remarks}

Here is an example of an application of Theorem \ref{t3} with $\alpha = u$.

\begin{corollary}
\label{cor5}
Suppose exists a TD$(8,m)$ and a $(7m+1,7,1)$-BIBD and let $0 \leq t \leq m-1$.
Then there exists a $\{7,8\}$-GDD of type $7^{7m}{(7t+1)}^1$.
\end{corollary}

\begin{proof}
We apply Theorem \ref{t3} with $\ell = u = v = \alpha = 7$ and $K = \{7,8\}$. A projective  plane of order $7$ with a point deleted yields an $8$-GDD of type $7^8$, which is a $K$-GDD of type $u^{\ell} v^1$. 
A TD$(\ell,u)$ with $\alpha$ disjoint blocks is obtained from a affine plane of order $7$ (one parallel class yields the groups and a second parallel class yields
$\alpha$ disjoint blocks). 
A $(7m+1,7,1)$-BIBD, which exists by hypothesis, is an $(mu+1,K)$-PBD. We obtain a $\{7,8\}$-GDD of type $7^{7m}{(7t+1)}^1$.
\end{proof}

\begin{remark}
{\rm It is known (see \cite[\S II.3.1]{CD}) that $(7m+1,7,1)$-BIBDs exist for $m \equiv 0 \text{ or } 1 \bmod 6$, $m > 372$. Also, $TD(8,m)$ are known to exist for all $m > 74$. So Corollary \ref{cor5} can be applied for all $m \equiv 0 \text{ or } 1 \bmod 6$, $m > 372$.}
\end{remark}

\section*{Acknowledgement}
Thanks to Ian Wanless for bringing Rosenfeld's paper \cite{Rosenfeld} to my attention and thanks to Julian Abel for helpful comments.

\end{document}